\documentclass{sig-alternate}

\usepackage{lipsum}

\usepackage{graphicx}
\usepackage{algorithmic}
\usepackage{amssymb}
\usepackage{amsmath}
\usepackage{graphicx}

\usepackage{mathrsfs}

\usepackage{multicol}
\usepackage{longtable}

\usepackage{breqn}

\usepackage{hyperref}

\newcommand{\Q}{\mathbb{Q}}
\newcommand{\C}{\mathbb{C}}

\newcommand{\N}{\mathbb{N}}

\newcommand{\order}{\mathop{\mathrm{order}}\nolimits}

\newcommand{\numerator}{\mathop{\mathrm{numerator}}\nolimits}
\newcommand{\denominator}{\mathop{\mathrm{denominator}}\nolimits}

\def \bg #1 {\begin{tabular}{{#1}}}
\def \nd {\end{tabular}}

\newtheorem{theorem}{Theorem}
\newdef{definition}{Definition}

\newtheorem{corollary}{Corollary}

\newdef{example}{Example}

\newenvironment{algorithm}[1]{
  \begin{center}
    {\bf Algorithm: #1}\\*
    \begin{tabular}{|p{80mm}|} \hline
} {
 \\ \hline
 \end{tabular}
 \end{center}
}

\begin{document}
%
% --- Author Metadata here ---
\conferenceinfo{ISSAC}{'17 University of Kaiserslautern, Kaiserslautern, Germany}

\conferenceinfo{ISSAC'17,} {July 25--28, 2017, Kaiserslautern, Germany.}

\CopyrightYear{2017}

%\crdata{978-1-59593-904-3/08/07}

\title{Algorithmic Verification of Linearizability for Ordinary Differential Equations}
\numberofauthors{3}
\author{
\alignauthor
Dmitry~A.~Lyakhov\\
\affaddr{King Abdullah University of Science and Technology\\
Thuwal, Makkah Region, Saudi Arabia}\\
\href{mailto:dmitry.lyakhov@kaust.edu.sa}{dmitry.lyakhov@kaust.edu.sa}
\alignauthor
Vladimir~P.~Gerdt\\
\affaddr{
Joint Institute for Nuclear Research, Dubna, Russia and Peoples' Friendship University of Russia, Moscow, Russia}\\
\href{mailto:gerdt@jinr.ru}{gerdt@jinr.ru}
\alignauthor
Dominik~L.~Michels \\
\affaddr{King Abdullah University of Science and Technology\\
Thuwal, Makkah Region, Saudi Arabia}\\
\href{mailto:dominik.michels@kaust.edu.sa}{dominik.michels@kaust.edu.sa}
}\maketitle
\begin{abstract}
For a nonlinear ordinary differential equation solved with respect to the highest order derivative and rational in the other derivatives and in the independent variable, we devise two algorithms to check if the equation can be reduced to a linear one by a point transformation of the dependent and independent variables. The first algorithm is based on a construction of the Lie point symmetry algebra and on the computation of its derived algebra. The second algorithm exploits the differential Thomas decomposition and allows not only to test the linearizability, but also to generate a system of nonlinear partial differential equations that determines the point transformation and the coefficients of the linearized equation. The implementation of both algorithms is discussed and their application is illustrated using several examples.
\end{abstract}

%\category{I.1.4}{Symbolic and Algebraic Manipulation}{Applications}


\keywords{Algorithmic linearization test, determining equations, differential Thomas decomposition, Lie symmetry algebra, ordinary differential equations, point transformation, power series solutions.

\section{Introduction}
Solving nonlinear ordinary differential equations (ODEs) is one of the classical and practically important research areas in applied mathematics. In practice, such equations are mostly solved numerically or by approximate analytical methods since obtaining their explicit solution is usually very difficult or even impossible. One of the important approaches for solving a nonlinear ODE explicitly considers the existence of an invertible linearizing transformation of the variables and its construction. The reduction of a nonlinear ODE to a linear one makes its explicit integration much easier and often allows for obtaining an exact solution.

The linearization problem for a second-order ODE
\begin{equation}\label{2ndOrder}
y^{\prime\prime}+f(x,y,y^\prime)=0
\end{equation}
was solved by Lie~(\cite{Lie'1883a}, Sect.~1), who applied his general theory of
integration of ODEs by means of a group of point transformations. He proved that $f$ is at most cubic in $y^\prime$ for a linearizable equation and derived the necessary and sufficient conditions of linearizability. These conditions have the form of two explicit and easily verifiable equalities~\eqref{LieTest1} containing differential polynomials in the coefficients of $f$ as a polynomial in $y^\prime$:
\begin{equation}\label{LieTest}
 f=F_3(x,y)(y^\prime)^3+F_2(x,y)(y^\prime)^2+F_1(x,y)\,y^\prime+F_0(x,y)\,.
\end{equation}

Lie's ideas and methods were extended and applied to third-order equations $y^{\prime\prime\prime}=f(x,y,y^\prime,y^{\prime\prime})$ \cite{IbragimovMeleshko'05} and later to fourth-order equations $y^{{\prime\prime\prime\prime}}=f(x,y,y^\prime,y^{\prime\prime},y^{\prime\prime\prime})$ \cite{IbragimovMeleshkoSuksern'08}.
In these contributions, all possible structures of the candidates for the linearization were found and the explicit form of necessary and sufficient linearizability conditions of the coefficients of those structures were derived. Therefore, given an
ODE of second or third order, to check whether  it is linearizable by a point transformation or not, it is sufficient to verify whether the relevant explicit linearizability conditions are satisfied or not. In practice, such a verification typically needs a computer-based symbolic algebraic computation for the simplification of the resulting expressions. An additional point to emphasize is that if the ODE contains parameters and/or arbitrary functions, then the linearizability
conditions imply the algebraic and/or differential constraints on these parameters and/or functions that provide the linearization. Generally, however, these constraints may include the point transformation functions, and it may be highly conjectural to solve the constraints and to find a linearizing point transformation.

In the present paper we suggest two algorithmic linearization tests applicable to a quasi-linear ODE (solved for the highest derivative) of any order greater or equal to two with a rational dependence on the other derivatives and the independent variable. The first linearization test is applicable to ODEs which do not contain parameters and arbitrary functions. This test is based on the construction of the Lie point symmetry algebra for the input ODE. The relevant mathematical methods are described in several textbooks (see, for example,~\cite{Ovsyannikov'92}--\cite{Ibragimov'09}). To detect linearizability
we compute the maximal abelian dimension of the Lie symmetry algebra and make use of the results of Mahomed and Leach~\cite{MahomedLeach'90}. Unlike the first test, our second test exploits the differential Thomas decomposition~(\cite{Thomas'37}--\cite{Robertz'14}), an universal algorithmic tool for the algebraic analysis of polynomially-nonlinear systems of partial differential equations (PDEs), and allows not only for the detection of linearizability but also for the derivation of necessary and sufficient linearizability conditions for arbitrary functions or parameters occurring in these equations. An example of such a problem is given by Eq.~\eqref{2ndOrder} whose linearizability conditions are given by Eq.~\eqref{LieTest}. Therefore, the second test can reproduce the above mentioned results of~\cite{Lie'1883a}--\cite{IbragimovMeleshkoSuksern'08}. Besides, via the second linearization test one can generate differential equations for a linearizing point transformation and for the coefficients of the linearized equation that are suitable for finding the transformation and the coefficients. However, the first linearization test is computationally more efficient and it is therefore advisable to apply it first when considering higher-order equations, and then, in the case of linearizability, apply the second test in order to construct the linearizing point transformation and the reduced linear form of the ODE.

This paper is organized as follows. In Sect.~2 we briefly describe the mathematical objects we deal with before presenting our algorithms in Sect.~3. The implementation of these algorithms in Maple is then described in Sect.~4 and its application is illustrated in Sect.~5 using several examples. Finally, we provide a conclusion in Sect.~6.

\section{Underlying Equations}
In this paper we consider ODEs of the form\\[-0.3cm]
\begin{equation}\label{ode}
y^{(n)}+f(x,y,y^{\prime},\ldots,y^{(n-1)})=0\,,\quad y^{(k)}:=\frac{d^ky}{dx^k}
\end{equation}
with $f\in \C(x,y,y^{\prime},\ldots,y^{(n-1)})$\footnote{In the subsequent, everywhere where it
is necessary from the computational point of view, the field $\Q$ is assumed to be considered instead of the field $\C$.} solved with respect to the highest order derivative. As additional arguments, the function $f$ may also include parameters and/or arbitrary functions in $x$ and/or $y$. Given an ODE of the form \eqref{ode}, our aim is to check the
existence of an invertible transformation\footnote{Hereafter, all functions we deal with are assumed to be smooth.}
\begin{equation}\label{transformation}
u=\phi(x,y)\,,\quad t=\psi(x,y)
\end{equation}
which maps \eqref{ode} into a linear $n$-th order homogeneous equation
\begin{equation}\label{lode}
u^{(n)}(t)+\sum_{k=0}^{n-1} a_{k}(t)\,u^{(k)}(t)=0\,,\quad u^{(k)}:=\frac{d^ku}{dt^k}\,.
\end{equation}
The invertibility of~\eqref{transformation} is provided by the inequation
\begin{equation}\label{Jacobian}
 J:=\phi_x \psi_y-\phi_y \psi_x\neq 0\,.
\end{equation}
If such a transformation exists for $n\geq 3$, then it can always be chosen (cf.~\cite{Olver'95}, Thm.~6.54;
\cite{Ibragimov'09}, Thm.~6.6.3) in a way that~\eqref{lode} takes the Laguerre-Forsyth normal form
\begin{equation}\label{lf-form}
u^{(n)}(t)+\sum_{k=0}^{n-3} a_{k}(t)\,u^{(k)}(t)=0\,.
\end{equation}
A first-order ODE $y^\prime=f(x,y)$ is always linearizable, but its linearization procedure is as hard as the integration of the equation (cf.\cite{Arnold'92},\,Ch.~2,\,Thm.~1).  For $n=2$ any homogeneous linear equation
\begin{equation*}\label{2ord_lode}
 y^{\prime\prime}(x)+a(x)y^\prime(x)+b(x)y(x)=0
\end{equation*}
can be transformed by a substitution
\[
   t=\varphi(x),\quad \varphi^\prime(x)\neq 0,\quad u=\sigma(x)y,\quad \sigma(x)\neq 0
\]
to the simplest second order equation~(\cite{Ibragimov'09},\,Thm.~3.3.1)
\begin{equation}\label{lf-form-of2ord-lode}
u^{\prime\prime}(t)=0\,.
\end{equation}

One way to check the linearizability of Eq.~\eqref{ode} is to follow the classical approach by
Lie~\cite{Lie'1883a} to study the symmetry properties of Eq.~\eqref{ode} under the {\em infinitesimal}
transformation
\begin{equation}\label{inf_trans}
\tilde{x}=x+ \varepsilon\,\xi(x,y) + \mathcal{O}(\varepsilon^2)\,,\ \ \tilde{y}=y + \varepsilon\,\eta(x,y) +
\mathcal{O}(\varepsilon^2)\,.
\end{equation}
The {\em invariance condition} for Eq.~\eqref{ode} under the transformation~\eqref{inf_trans} is given by the
equality
\begin{equation}\label{invariance}
 {\cal{X}}(y^{(n)}+f(x,y,...,y^{(n-1)}))|_{y^{(n)}+f\left(x,y,\ldots,y^{(n-1)}\right)=0}=0,
\end{equation}
where the {\em symmetry operator} reads
\begin{equation}\label{symm_generator}
{\cal{X}}:=\xi\,{\partial_x}+\sum_{k=0}^n\eta^{(k)}\partial_{y^{(k)}}\,,\
\eta^{(k)}:=D_x\eta^{(k-1)}-y^{(k)}D_x\,\xi,
\end{equation}
$\eta^{(0)}:=\eta$ and $D_x:=\partial_x+\sum_{k\geq 0}y^{(k+1)}\partial_{y^{(k)}}$
is the total derivative operator with respect to $x$.

The invariance condition~\eqref{invariance} means that its left-hand side vanishes when
Eq.~\eqref{ode} holds. Then the application of~\eqref{symm_generator} to the left-hand side of Eq.~\eqref{ode}
and the substitution of $y^{(n)}$ with $-f(x,y,\ldots,y^{(n-1)})$ in the resulting expression leads to the equality $g=0$
with the polynomial dependence of $g$ on the derivatives $y^\prime,\ldots,y^{(n-1)}$. Since, by
Def.~\eqref{inf_trans}, the functions  $\xi$ and $\eta$ do not depend on these derivatives, the equality $g=0$
holds if and only if all coefficients in $y^\prime,\ldots,y^{(n-1)}$ are equal to zero. This leads to an
overdetermined system of linear PDEs in $\xi$ and $\eta$ called {\em determining system}. Its solution yields a
set of symmetry operators whose cardinality we denote by $m$. This set forms a basis of the $m$-dimensional {\em Lie
symmetry algebra}
\begin{equation}\label{LieAlgebra}
[{\cal{X}}_i,{\cal{X}}_j]=\sum_{k=1}^{m} C^k_{i,j}{\cal{X}}_k\,,\quad 1\leq i<j\leq m\,.
\end{equation}
We denote the Lie symmetry algebra by $L$ and $m=\dim(L)$.  Its {\em derived algebra} $L^\prime\subset L$
is a subalgebra that consists of all commutators of pairs of elements in $L$.

Lie showed~(\cite{Lie'1883b}, Ch.~12, p.~298, ``Satz'' 3) that the Lie point symmetry algebra of an $n$-order
ODE has a dimension $m$ satisfying
\begin{equation*}\label{inequality}
 n=1,\ m=\infty;\ \ n=2,\ m\leq 8;\ \ n\geq 3,\ m\leq n+4\,.
\end{equation*}

Later, the interrelations between $n$ and $m$ were established that provide the linearizability of~\eqref{ode}
by a point transformation~\eqref{transformation} in the absence of parameters and arbitrary functions. Here
we present the two theorems that describe such interrelations and form the basis of our first linearization test.
\begin{theorem}{\em(\cite{MahomedLeach'90}, Thm.~1)}
A necessary and sufficient condition for the linearization of~\eqref{ode} with $n\geq 3$ via a point
transformation is the existence of an abelian $n$-di\-men\-sional subalgebra of~\eqref{LieAlgebra}.
\label{theorem1}
\end{theorem}
\begin{theorem}{\em(\cite{MahomedLeach'90}, Sect.~2--4 and Thms.~6,8; \cite{Schwarz'08}, Thm.~5.19)}\footnote{Cf.
Thm.~6.39 in~\cite{Olver'95} regarding $n=2$.} Eq.~\eqref{ode} with $n\geq 2$ is linearizable by
a point transformation if and only if one of the following conditions is fulfilled:
\begin{enumerate}
\item $n=2$, $m=8$;
\item $n\geq 3$, $m=n+4$;
\item $n\geq 3$, $m\in \{n+1,n+2\}$ and~\eqref{LieAlgebra} admits an abelian subalgebra of dimension $n$.
\end{enumerate}\label{theorem2}
\end{theorem}

These theorems show that the verification of the third condition requires, in addition to the determination of $m$,
a computation to check the existence of an abelian Lie symmetry subalgebra of dimension $n$. To our knowledge,
there is only an algorithm described in the literature~\cite{Ceballos'12} for the computation of
the {\em maximal abelian dimension}, i.e.~dimension of the maximal abelian subalgebra of a finitely-dimensional
Lie algebra given by its structure constants. The algorithm is reduced to solving the quadratically nonlinear
system of multivariate polynomial equations providing vanishing of the Lie bracket between two arbitrary
vectors in the Lie algebra. Clearly, the runtime of the algorithm is at least exponential in
the dimension $m$ of the algebra.

Instead, to verify the third condition in Thm.~\ref{theorem2} we devise a much more efficient algorithm. Our algorithm relies on the following statement which is a corollary to Thms.~\ref{theorem1} and~\ref{theorem2}.

\begin{corollary}\label{corollary}
The third condition is equivalent to
\begin{enumerate}
\item[3'.] $n\geq 3$, $m\in \{n+1,n+2\}$ and the derived algebra of~\eqref{LieAlgebra} is abelian of dimension $n$.
\end{enumerate}

\end{corollary}

\begin{proof}
Under the third condition, since $L^\prime \subset L$, Eq.~\eqref{ode} is linearizable by Thm.~1. Let Eq.~\eqref{ode}
be linearizable. The symmetry Lie algebra of~\eqref{lode} is {\em similar} and hence {\em isomorphic} to that of~\eqref{ode} ~(cf.~\cite{Ovsyannikov'92}, Ch.~2, \textsection 7.9). It is easy to see that a linear $n$-th order equation~\eqref{lode} with variable coefficients admits the Lie point symmetry group
\begin{equation}\label{group}
\{\,\bar{t}=t,\ \bar{u}=u+c_i\cdot v_i(t)\ (i=1,..,n),\  \bar{u}=c_{n+1}\cdot u\,\},
\end{equation}
where $c_i,c_{n+1}$ are constants (the group parameters) and $v_i(t)$ are the fundamental solutions of~\eqref{lode}.
The Lie group~\eqref{group} has the $(n+1)$-dimensional Lie algebra (cf.~\cite{Schwarz'08}, Thm.~5.19)
\begin{equation}\label{LA1}
L_{n+1}:=\{\,{\cal{X}}_i:=v_i(t)\,\partial_u\  (i=1,..,n),\  {\cal{X}}_{n+1}:=u\,\partial_u\,\}\,.
\end{equation}
If a linear $n$-th order Eq.~\eqref{lode} has constant coefficients, then in addition to~\eqref{LA1}
the Lie point symmetry group~\eqref{group} includes the translation $\bar{t}=t+c_{n+2}$ and, hence, its Lie
algebra, in addition to~\eqref{LA1},  has one more element:
\begin{equation}\label{LA2}
L_{n+2}:=L_{n+1}\cup \{{\cal{X}}_{n+2}:=\partial_t\}\,.
\end{equation}
Furthermore, $[{\cal{X}}_{n+1},{\cal{X}}_{n+2}]=0$, and for all $i\in \{1,\ldots,n\}$:
\[
[{\cal{X}}_i,{\cal{X}}_{n+1}]={\cal{X}}_i,\  [{\cal{X}}_i,{\cal{X}}_{n+2}] =
- v^\prime_i(t)\partial_u=\sum_{j=1}^n \alpha_j{\cal{X}}_j,\
\]
where $\alpha_i$ are constants. Therefore, both Lie algebras~\eqref{LA1} and~\eqref{LA2} have abelian
derived algebras of dimension $n$.
\end{proof}
It is important to emphasize that $m$ can be algorithmically computed without solving the determining system
what is generally impossible. It suffices to complete the last system to involution (for the theory of
completion to involution we refer to~\cite{Seiler'10}) and to construct power series solutions to the
involutive system~\cite{Reid'91a,Reid'91b}. For instance, as we do in our implementation (Sect.~4) of the
algorithm {\textsl{\bfseries{LinearizationTest~I}} described in Sect.~3.1, one can apply to the determining
system the {\em differential Thomas decomposition}~\cite{BGLHR'12,Robertz'14} for a
{\em degree-reverse lexicographical ranking} and then compute the {\em differential dimension polynomial}~\cite{Lange-Hegermann'14} for the output Janet basis.

The differential Thomas decomposition was suggested in \cite{Thomas'37,Thomas'62} as a
generalization of the Riquier-Janet theory of {\em passive} linear and {\em orthonomic} PDE systems (see also~\cite{Seiler'10}
and the references therein) to polynomially-nonlinear systems of general form. The Thomas decomposition provides a
universal algorithmic tool~\cite{BGLHR'12,Robertz'14} to study a {\em differential system}, which is defined as
follows.

\begin{definition}$($\cite{Thomas'37}--\cite{Robertz'14}$)$
A {\em differential system} is a system  $S:=\{S^{=},S^{\neq}\}$ of differential equations and (possibly) inequations
of the form
\[
   S^{=}:=\{g_1=0,\ldots,g_s=0\},\ S^{\neq}:=\{h_1\neq 0,\ldots,h_t\neq 0\},
\]
where $s$ is a positive integer as well as $t$ if $S^{\neq}\neq \emptyset$, and $q_i,h_j$ are elements in the differential polynomial ring in finitely many
differential indeterminates (dependent variables) over the differential field of characteristics zero.
\label{DS}
\end{definition}

The Thomas decomposition applied to a differential system $S$ yields a finite set of passive (involutive)
and {\em differentially triangular} differential systems called
{\em simple}~(see~\cite{Thomas'37}-\cite{Robertz'14}) that partition the solution set  of the input
differential system. Algebraically, this provides a characterizable decomposition~\cite{Hubert'03} of
the radical differential ideal $\sqrt{{\cal{I}}}$ where ${\cal{I}}$ is the differential ideal generated
by the polynomials in $S^{=}$.

Unlike the {\textsl{\bfseries{LinearizationTest I}} where one can use, due to the linearity of determining
systems, any procedure of completion to involution (e.g.~the standard form algorithm~\cite{Reid'91a}), our
second algorithm {\textsl{\bfseries{LinearizationTest II}} (Sect.~3.2) is oriented to the Thomas
decomposition.

To apply it, we need to formulate the conditions for the functions $\phi(x,y)$, $\psi(x,y)$
in~\eqref{transformation} and for the coefficients $a_k(t)$ in~\eqref{lf-form} (if $n\geq 3$)
such that these conditions hold if and only if~\eqref{ode} is linearizable. In addition to the
input differential system, the Thomas decomposition is determined by a {\em ranking}, that is,
a linear ordering on the partial derivatives compatible with derivations~(\cite{Thomas'37}--\cite{Robertz'14})
(in our case with $\partial_x$ and $\partial_y$).

By differentiating the equality $u(\psi(x,y(x)))=\phi(x,y(x))$, that follows from~\eqref{transformation},
$n$ times with respect to $x$, we obtain the following equalities:
\begin{eqnarray}\label{diff-u}
&& u^\prime(t)=\frac{\phi_x+\phi_yy^\prime}{\psi_x+\psi_yy^\prime}\,,\nonumber \\
&& u^{\prime\prime}(t)=\frac{J}{(\psi_x+\psi_yy^\prime)^3}\,y^{\prime\prime}+
\frac{P_2(y^\prime)}{(\psi_x+\psi_yy^\prime)^3}\,,\\
&& \hspace{0.5cm}\vdots \nonumber \\
&& u^{(n)}(t)=\frac{J}{(\psi_x+\psi_yy^\prime)^{n+1}}\,y^{(n)}+\frac{P_n(y^\prime,\ldots,y^{(n-1)})}{(\psi_x+
\psi_yy^\prime)^{2n-1}}\,. \nonumber
\end{eqnarray}
Here $J$ is the Jacobian~\eqref{Jacobian}, $P_k$ $(k=(2,\ldots,n))$ are polynomials in their arguments whose
coefficients are differential polynomials in $\phi$ and $\psi$, for example,
\begin{dmath*}
P_2(y^\prime)=(\psi_x+\psi_yy^\prime)\left(\phi_{xx}+\phi_{xy}y^\prime+\phi_{yy}(y^\prime)^2\right)-
(\phi_x+\phi_yy^\prime)\left(\psi_{xx}+\psi_{xy}y^\prime+\psi_{yy}(y^\prime)^2\right)\,.
\end{dmath*}

Now we replace the derivatives $u^{(k)}$ occurring in~\eqref{lf-form} (or the second-order derivative
in~\eqref{lf-form-of2ord-lode} if $n=2$) with the appropriate right-hand sides in Eqs.~\eqref{diff-u} and
solve the obtained equality with respect to $y^{(n)}$ (or $y^{\prime\prime}$). As a result, we obtain
the equality
\begin{equation}\label{ode-subs}
 y^{(n)}+ \frac{R(y^\prime,\ldots,y^{(n-1)})}{J\cdot (\psi_x+
\psi_yy^\prime)^{n-2}}=0\,,
\end{equation}
where $R$ is a polynomial in $y^\prime,\ldots,y^{(n-1)}$ whose coefficients for $n\geq 3$ are the
differential polynomials not only in $\phi$ and $\psi$ but also in $a_k(t)=a_k(\psi(x,y(x)))$, the
coefficients in Eq.~\eqref{lf-form}.

Denote by $M$ and $N$ the numerator and denominator of the function $f$ in Eq.~\eqref{ode}. Then, after
elimination of $y^{(n)}$ from the equation system~\eqref{ode},~\eqref{ode-subs} and clearing denominators
in the rational functions of the obtained equality we obtain equation
\begin{equation}\label{pol-eq}
  R\cdot N - J\cdot M\cdot  (\psi_x+\psi_yy^\prime)^{n-2}=0\,.
\end{equation}

This equation is a polynomial in $y^\prime,y^{\prime\prime},\ldots,y^{(n-1)}$, and there are no constraints
on these variables. Therefore, the equation holds if and only if all coefficients of the polynomial in the
left-hand side vanish. This condition gives a partial differential equation system in $\phi$, $\psi$ and $a_k$. If the
function $f$ in Eq.~\eqref{ode} depends on parameters and/or undetermined functions in $(x,y)$,
then Eq.~\eqref{pol-eq} contains these parameters/functions.\footnote{One can always consider parameters as
functions in $x$ and $y$ with zero derivatives.}

Let $S^{=}$ be the set of equations obtained from Eq.~\eqref{pol-eq} by equating the coefficients of the
polynomial (in $y^\prime,\ldots,y^{(n-1)}$) in the left-hand side to zero. If $n\geq 3$ we enlarge $S^{=}$
with the set of equations
\begin{equation}\label{PDE-for-a}
  S^{=}=S^{=}\cup_{k=0}^{n-3} \{\psi_y(a_k)_x - \psi_x(a_k)_y=0\}.
\end{equation}
The equation $\psi_y(a_k)_x - \psi_x(a_k)_y=0$ means that $a_k$ is a function of $t$ in accordance
to~\eqref{lf-form}. It is easy to see by differentiating the equality $a_k(t)=a_k(\psi(x,y))$ as follows:
\[
  (a_k)_x=a_t\psi_x,\  (a_k)_y=a_t\psi_y,\ \Longleftrightarrow\ \psi_y(a_k)_x - \psi_x(a_k)_y=0\,.
\]
Since we admit the invertible transformations~\eqref{transformation} only, one has to add to the
enlarged equation set $S^{=}$ the inequation
\begin{equation}\label{inequation}
S^{\neq}:=\{J\neq 0\}
\end{equation}
where $J$ is the Jacobian~\eqref{Jacobian}.

Thereby, the main object of our construction and the statements on its relation to the
linearization are given as follows.

\begin{definition}\label{LDS} The differential system (see Def.~\ref{DS}) made up of the above
constructed PDE set $S^{=}$ and of the inequation set $S^{\neq}=\{J\neq 0\}$ will be called
{\em linearizing differential system}.
\end{definition}

\begin{theorem}\label{LinTD} Eq.~\eqref{ode} is linearizable via a point
transformation~\eqref{transformation} if and only if the linearizing differential system is
consistent, i.e.~has a solution.
\end{theorem}

\begin{corollary}\label{nonzero}
 Eq.~\eqref{ode} is nonlinearizable via a point transformations~\eqref{transformation} if and
 only if the result of the Thomas decomposition algorithm~(\cite{BGLHR'12}, Alg.~2.25; \cite{Robertz'14},
 Alg.~2.2.56) applied to the linearizing differential system is the empty set.
\end{corollary}

\begin{proof} \cite{BGLHR'12}, Remark 2.3 and~\cite{Robertz'14}, Remark 2.2.58. \end{proof}

\section{Linearization Tests}
In this section we present our algorithms {\textsl{\bfseries{LinearizationTest I}} and
{\textsl{\bfseries{LinearizationTest II}}. These algorithms, given an input equation~\eqref{ode},
verify its linearizability by the point transformation~\eqref{transformation}. In so doing,
the first test is applicable only to an ODE without parameters and undetermined functions
in the variables $x$ and $y$. The second algorithm admits a rational dependence of the function $f$ in Eq.~\eqref{ode}
on such parameters and functions.

\subsection{Linearization test I}

Our first test, presented below, is based on the computation of the Lie symmetry algebra and its analysis.
In line 2 we compute the determining system for~\eqref{ode}. It is the straightforward procedure outlined in
the preceding section and described in most textbooks on Lie symmetry analysis, in particular
in~\cite{Ovsyannikov'92}--\cite{Ibragimov'09}. As a routine, this procedure is present in most computer algebra
packages specialized to such an analysis, for example, in the Maple packages~{\tt DESOLV}~\cite{CarminatiVu'00},
{\tt DESOLVII}~\cite{VuJefferesonCarminati'12}, and {\tt SADE}~\cite{FilhoFigueiro'11}.

Since the determining system is linear, one can use any algorithm for its completion to involution in line 3,
(cf.~\cite{Reid'91a} and~\cite{Seiler'10}, Sect.~10.7). However, we prefer to use the differential Thomas
algorithm here (\cite{BGLHR'12}, Sect.~3 and~\cite{Robertz'14} Sect.~2.2).

The dimension $m$ of the Lie algebra~\eqref{LieAlgebra} (line 4) is the dimension of the solution space of
the determining system and can be computed in several ways (cf.~\cite{Seiler'10}, Sect.~8.2 and 9.3). Having computed the Janet involutive form of the determining system, it
is easy to compute the dimension of its solution via an algorithmic construction of the {\em differential
dimension polynomial}~\cite{Lange-Hegermann'14}.

\begin{algorithm}{\textsl{\bfseries{LinearizationTest I}}\,($q$)
\label{LinearizationTest-1}}
\begin{algorithmic}[1]
\INPUT $q$, a nonlinear differential equation of form~\eqref{ode}
\OUTPUT {\tt True}, if $q$ is linearizable and $\tt False$, otherwise
\STATE $n:=\order(q)$;
\STATE $DS:=$ {\textsl{\bfseries{DeterminingSystem}}}\,($q$);
\STATE $IDS:=$ {\textsl{\bfseries{InvolutiveDeterminingSystem}}}\,($DS$);
\STATE $m:=\dim$({\textsl{\bfseries{LieSymmetryAlgebra}}})\,($IDS$);
\IF{$n=1 \vee (n=2 \wedge m=8) \vee (n>2 \wedge m=n+4) $}
   \RETURN {\tt True};
\ELSIF{$n>2 \wedge (m=n+1 \vee m=n+2)$}
   \STATE $L:=$ {\textsl{\bfseries{LieSymmetryAlgebra}}}\,($IDS$);
   \STATE $DA:=$ {\textsl{\bfseries{DerivedAlgebra}}}\,($L$);
   \IF{$DA$ is abelian and $\dim(DA)=n$}
      \RETURN {\tt True};
   \ENDIF
\ENDIF
\RETURN {\tt False};
\end{algorithmic}
\end{algorithm}

We refer to~\cite{Reid'91b} for the subalgorithm providing computation of the Lie symmetry algebra
(line 8), i.e.~for the computation of the structure constants $C_{i,j}^k$ in Eq.~\eqref{LieAlgebra}.
The last subalgorithm~{\textsl{\bfseries{DerivedAlgebra}}}  in line 9 does the
straightforward computation of the derived algebra via the structure constants.

{\em Correctness and termination}. For the subalgorithms both these properties are either obvious
(as for {\textsl{\bfseries{DerivedAlgebra}}}) or shown in the papers we referred to in the description of the subalgorithms above. Therefore, the whole algorithm {\textsl{\bfseries{LinearizationTest I}}
terminates, and its correctness is provided by Thms.~\ref{theorem1} and \ref{theorem2}, and
Cor.~\ref{corollary}.

\subsection{Linearization test II}

Our second test is based on the differential Thomas decomposition~\cite{BGLHR'12,Robertz'14}.
It admits the rational dependence of Eq.~\eqref{ode} on a finite set of parameters (constants) and/or
undetermined functions in $(x,y)$. In the absence of parameters/functions the corresponding sets
are inputted as the empty ones.

\begin{algorithm}{\textsl{\bfseries{LinearizationTest II}}\,($q,P,H$)
\label{LinearizationTest-2}}
\begin{algorithmic}[1]
\INPUT $q$, a nonlinear differential equation of form~\eqref{ode} of order $\geq 2$; P, a set of
parameters; H, a set of undetermined functions in $(x,y)$
\OUTPUT Set $G$ of differential systems for functions $\phi$ and $\psi$ in~\eqref{transformation} and
(possibly) in elements of $P$ and $H$ if~\eqref{ode} is linearizable, and the empty set, otherwise
\STATE $n:=\order(q)$;
\STATE $G:=\emptyset$;
\STATE $M:=\numerator(f)$;\ \ $N:=\denominator(f)$;
\STATE $J:=\phi_x\psi_y-\phi_y\psi_x$; \COMMENT{Jacobian~\eqref{Jacobian}}
\IF{$n=2$}
  \STATE $r:=u^{\prime\prime}(t)=0$; \COMMENT{ODE~\eqref{lf-form-of2ord-lode}}
  \STATE $A:=\emptyset$;
\ELSE
  \STATE $r:=u^{(n)}(t)+\sum_{k=0}^{n-3}\,a_k(t)u^{(k)}(t)=0$; \COMMENT{ODE~\eqref{lf-form}}
  \STATE $A:=\{a_0,\ldots,a_{n-3}\}$;
\ENDIF
%\STATE $U:=\{\phi,\psi\}\cup A\cup P\cup H$; \COMMENT{differential indeterminates}
%\STATE $\K:=\Q\{U\}$ \COMMENT{differential polynomial ring}
%\STATE $r:=y^{(n)}+\frac{R(y^\prime,\ldots,y^{(n)})}{J\cdot (\psi_x+\psi_yy^\prime)^{(n-2)}}=0$; \COMMENT{Eq.~\eqref{ode-subs}}
\STATE $r \xrightarrow{\text{by\ }\eqref{transformation}}y^{(n)}+\frac{R(y^\prime,\ldots,y^{(n-1)})}{J\cdot (\psi_x+\psi_yy^\prime)^{(n-2)}}=0$; \COMMENT{Eq.~\eqref{ode-subs}}

%\COMMENT{$R\in \C\{\phi,\psi,A\}[y^\prime,\ldots,y^{(n-1)}]$}

\STATE $T:=R\cdot N-M\cdot J\cdot (\psi_x+\psi_yy^\prime)^{(n-2)}=0$; \COMMENT{Eq.~\eqref{pol-eq}}
\STATE $S^{=}:=\{c=0 \mid c\in {\text{coeffs}}\,(T,\{y^{\prime},\ldots,y^{(n-1)}\}) \}$;

%\COMMENT{PDE system for the elements in $U$}
\STATE $S^{=}:=S^{=}\cup_{p\in P}\{p_x=0,p_y=0\}$;
\STATE $S^{=}:=S^{=}\cup_{a\in A} \{a_x\psi_y-a_y\psi_x=0\}$; \COMMENT{Eq.~\eqref{PDE-for-a}}
\STATE $S^{\neq}:=\{J\neq 0\}$; \COMMENT{Ineq.~\eqref{inequation}}
\STATE $G:=$ {\textsl{\bfseries{ThomasDecomposition}}}\,($S^{=},S^{\neq}$);
\RETURN $G$;
\end{algorithmic}
\end{algorithm}

In lines 3--17 of the algorithm {\textsl{\bfseries{LinearizationTest II}} the input linearizing
differential system (Def.~\ref{LDS}) is constructed for the Thomas decomposition computed in
line 18.  This construction is done in correspondence with the
formulas~\eqref{Jacobian}--\eqref{lf-form}, \eqref{lf-form-of2ord-lode}, and \eqref{ode-subs}--\eqref{inequation}.
Furthermore, if the output set of the Thomas decomposition is nonempty, then Eq.~\eqref{ode} is
linearizable by Thm.~\ref{LinTD}. In this case the simple systems in the decomposition provide a
partition of the solution space of the linearizing differential system and their solutions determine
the invertible point transformation~\eqref{transformation} and the coefficients $a_k(t)$ of the
linearized form~\eqref{lf-form} or~\eqref{lf-form-of2ord-lode}. In addition, if there are parameters
and/or undetermined functions in~\eqref{ode}, then the output differential systems of the Thomas decomposition
provide the compatibility conditions to these parameters/functions imposed by the linearization.

{\em Correctness and termination} are provided by those of the Thomas
decomposition~(\cite{BGLHR'12}, Sect.~3.4; \cite{Robertz'14}, Thr. 2.2.57).

\section{Implementation}

We implemented both linearization tests in Maple. Our implementation runs on Version 16 and the
subsequent ones.

First, we describe our implementation of the
algorithm {\textsl{\bfseries{LinearizationTest I}}. Given an ordinary differential equation of
the form~\eqref{ode}, to generate the determining system, denoted by $DS$ in line 3, we use the
routine {\em gendef} of the Maple package {\tt DESOLV}~\cite{CarminatiVu'00,VuJefferesonCarminati'12}. Then,
to complete the system $DS$ to involution (line 3), we choose the {\em orderly (``DegRevLex'') ranking}
(cf.~\cite{Robertz'14}, Def.~A.3.2) such that
\[
  \partial_x\succ \partial_y\,,\quad \xi\succ \eta\,,
\]
and apply the routine~{\em DifferentialThomasDecomposition} of
the package {\tt DifferentialThomas}. This package is freely available~\cite{DTD}. To compute the
dimension of the Lie symmetry algebra (line 4), we invoke the routine {\em DifferentialSystemDimensionPolynomial}.
Since in our case the solution space of the determining system is finitely dimensional, the last routine outputs
just the dimension of the solution space.

The subalgorithm {\textsl{\bfseries{LieSymmetryAlgebra}}} of line 8 was implemented in Maple
(see~\cite{Reid'91b}, Sect.~6). The implementation is based on the one of two other algorithms: the
standard form algorithm for completion of the determining system to involution and on the algorithm
of calculating power series solutions~\cite{Reid'91a}. Since that implementation done in Maple V has
not been adopted to the subsequent versions of Maple, we decided to make our own implementation of the
algorithmic approach suggested in~\cite{Reid'91b} to compute the structure
constants in~\eqref{LieAlgebra}. Our implementation takes the Janet involutive form of the determining system
outputted by the package {\tt DifferentialThomas} and exploits its routine {\em PowerSeriesSolution}.

To construct the derived algebra (line 9) we invoke the routine {\em DerivedAlgebra} which is a part of
the built-in package {\tt DifferentialGeometry:-LieAlgebras}.

In our implementation of {\textsl{\bfseries{LinearizationTest II}} we compute the expressions~\eqref{diff-u}
to obtain the left-hand side in~\eqref{pol-eq} (line 13) that is a polynomial in
$y^\prime,y^{\prime\prime},\ldots,y^{(n)}$. Then equating of all coefficients in the polynomial to zero (line 14)
and enlarging it with additional equations (lines 15 and 16) and the Jacobian inequation (line 17) yields the input linearizing
differential system for the subroutine {\em differentialThomasDecomposition} (line 18). By default,
we choose the orderly ranking
on the partial derivatives of the functions $\phi$ and $\psi$, and of those in the sets $A$ (line 10):
\[
 \partial_x\succ \partial_y\,,\quad \xi\succ \eta\succ a_0\succ\cdots\succ a_{n-2}\,.
\]

If the input ODE~\eqref{ode} contains (nonempty) sets of parameters and/or undetermined functions, then their
rankings are less than those of $a_{n-2}$ in order to derive the compatibility conditions for
parameters and functions.

\section{Examples}

In this section we demonstrate our algorithmic linearization tests
using several examples. All timings given below were obtained with Maple 16 running on a desktop computer with an Intel(R)Xeon(R) X5680 CPU clocked at 3.33 GHz and 48\,GB RAM.

\begin{example}$\cite{Lie'1883a}$
Consider the second-order Eq.~\eqref{2ndOrder} in which $f$ is given by~\eqref{LieTest} with undetermined
functions $F_k$, $0\leq k\leq 3$. Algorithm {\textsl{\bfseries{LinearizationTest I}}} is not applicable
to this case, so we apply the algorithm {\textsl{\bfseries{LinearizationTest II}}} with an
orderly ranking such that
\[
  \partial_x\succ \partial_y\,,\quad \phi\succ \psi\succ F_3 \succ F_2 \succ F_ 1\succ F_0.
\]
Then the routine {\em DifferentialThomasDecomposition} of the package {\tt DifferentialThomas}~\cite{DTD}
outputs three differential systems with disjoint solutions space in about 0.4~sec.:
\[
  S_1:=\{S_1^{=},S_1^{\neq}\}\,,\quad S_2:=\{S_2^{=},S_2^{\neq}\}\,,\quad S_3:=\{S_3^{=},S_3^{\neq}\}\,.
\]
Cor.~\ref{nonzero} guarantees that there
are linearizable
equations among the equations in family~\eqref{2ndOrder}--\eqref{LieTest}.
One of the output differential systems, namely $S_1$,
is a {\em generic simple system}
(see~\cite{Robertz'14}, Def.~2.2.67). It has eight equations, and the last two of them that contain solely functions
$F_0,F_1,F_2,F_3$  are the
{\em compatibility conditions} for these functions whose solutions admit linearization.
 These conditions have the following form:
\begin{eqnarray}
3(F_3)_{xx} - 2(F_2)_{xy} + (F_1)_{yy} - 3F_1(F_3)_x + 2F_2(F_2)_x \nonumber \\
 - 3F_3(F_1)_x
 + 3F_0(F_3)_y + 6F_3(F_0)_y  - F_2(F_1)_y =0\,,\nonumber\\[-0.4cm]
 \label{LieTest1}\\
(F_2)_{xx} - 2(F_1)_{xy} + 3(F_0)_{yy} - 6F_0(F_3)_x + F_1(F_2)_x \nonumber \\
- 3F_3(F_0)_x + 3F_0(F_2)_y  + 3F_2(F_0)_y  -
2F_1(F_1)_y =0\,.\nonumber
\end{eqnarray}
These equations are exactly the linearizability conditions for~\eqref{2ndOrder}--\eqref{LieTest} obtained
by Lie in~\cite{Lie'1883a} (cf.~\cite{Ibragimov'09}, Thm.~6.5.2). The inequations in the $S_1$ system are
\[
S_1^{\neq}=\{J\neq 0,\psi_x\neq 0, \psi_y\neq 0\}\,.
\]
The two other differential systems $S_2$ and $S_3$ have the following inequations:
\begin{equation}\label{ineqJ}
S_2^{\neq}=\{\phi_x\neq 0,\psi_y\neq 0\}\,,\quad S_3^{\neq}=\{\phi_y\neq 0,\psi_x\neq 0\}\,.
\end{equation}
Each of these systems has eight equations as $S_1$. Every equation in $S_2$ as well as in $S_3$
is valid on all common solutions to the
equations in $S_1$ (cf.~\cite{Robertz'14}, Cor.~2.2.66). In doing so,
\begin{equation}
   \psi_x=0\in S_2^{=}\,,\quad \psi_y=0\in S_3^{=}\,, \label{eqJ}
\end{equation}
and hence each of~\eqref{ineqJ} and \eqref{eqJ} implies $J\neq 0$. Therefore,
algorithm {\textsl{\bfseries{LinearizationTest II}}} reproduces Lie's classical  results on the necessary
and sufficient conditions for the linearization of the second-order ODEs from
family~\eqref{2ndOrder}--\eqref{LieTest}.
\label{LieExample}
\end{example}

\begin{example}$(\cite{Wolf'03},\text{Eq.~2.50})$
We consider the third order ODE
\begin{equation}\label{3rdODE}
 y^{\prime\prime\prime}-6\frac{y^\prime}{x^2}+3\frac{(y^{\prime})^2}{x}-\frac{1}{2}(y^{\prime})^3=0\,.
\end{equation}
This equation is linearizable by the generalized Sundman
transformation\footnote{A kind of a nonlocal transformation, which is in general not a point transformation.}~\cite{Wolf'03}. Here we check its linearizability via the point transformation~\eqref{transformation}. Eq.~\eqref{3rdODE} admits both our tests since it does not contain parameters and/or undetermined functions. Our implementation of algorithm {\textsl{\bfseries{LinearizationTest I}}}
returns {\em false} in 0.05~sec.~and that of {\textsl{\bfseries{LinearizationTest II}}} returns the empty set
in 0.4~sec.
\label{3rdOrder}
\end{example}

\begin{example}
We consider the fourth-order ODE
\begin{dmath}\label{IM}
 2x^2y\,y^{\prime\prime\prime\prime}+x^2y^2+h(x,y)\,y^{\prime}y^{\prime\prime\prime}+
     16x\,y\,y^{\prime\prime\prime}+6x^2(y^{\prime\prime})^2+
  48x\,y^{\prime}y^{\prime\prime}+24y\,y^{\prime\prime}+24(y^{\prime})^2=0\,,
\end{dmath}
where $h(x,y)$ is an undetermined function. To find all values of this function providing linearization,
we again apply algorithm {\textsl{\bfseries{LinearizationTest II}}}. The package {\tt DifferentialThomas}
for the orderly ranking satisfying
\begin{equation*}\label{Exm4ranking}
  \partial_x\succ \partial_y\,,\quad \phi\succ \psi\succ a_0 \succ a_1 \succ b
\end{equation*}
outputs in 3.3~sec.~two differential systems $S_1$ and $S_2$ (see~\eqref{EqS1} and~\eqref{EqS2}). Each
system has only one equation containing $h(x,y)$:
\begin{equation}\label{ParameterValue}
h(x,y)-8x^2=0\,.
\end{equation}
The linearizability of~\eqref{IM} under condition $\eqref{ParameterValue}$ was established
in~\cite{IbragimovMeleshkoSuksern'08}, and our computation shows that there are no other linearizable equations
of family~\eqref{IM}. Moreover, the simple systems $S_1$ and $S_2$ allow for the explicit construction of the  linearizing point transformation~\eqref{transformation} and the coefficients $a_0(t)$ and $a_1(t)$ in the  Lagerre-Forsyth form~\eqref{lf-form} of the image of~\eqref{IM} under mapping~\eqref{transformation}:
\begin{equation*}\label{LFforIM}
u^{(4)}(t)+ a(t)\,u(t)+b(t)u^{\prime}(t)=0\,.
\end{equation*}
To show this, consider first the equations in $S_1^{=}$:
\begin{equation}\label{EqS1}
\begin{array}{l}
y\,\phi_y-2\,\phi-2\,\phi_{xxxx}=0,\ \phi_{xxxy}=0,\\
x^2\phi_{xxy}-2\,\phi_y=0,\ x\,\phi_{xy}-2\phi_y=0,\\
y\phi_{yy}-\phi_y=0,\
a\,\psi_x^4-1=0,\ \psi_y=0,\\
a_x=0,\ a_y=0,\ b=0,\
h-8\,x^2=0\,,
\end{array}
\end{equation}
and its inequations
\begin{equation}\label{IneqS1}
S_1^{\neq}=\{a\neq 0,\phi_y\neq 0\}.
\end{equation}
The equation system~\eqref{EqS1} can easily be integrated by hand or using the Maple routine {\em pdsolve}. The general solution to~\eqref{EqS1}, in addition to~\eqref{ParameterValue}, reads
\[
\begin{split}
\phi_1:=&c_1x^2y^2+\sin\left(\frac{x}{\sqrt{2}}\right)\left(c_2\exp\left(-\frac{x}{\sqrt{2}}\right)+
     c_3\exp\left(\frac{x}{\sqrt{2}}\right)\right)\\
     &+\cos\left(\frac{x}{\sqrt{2}}\right)\left(c_4\exp\left(-\frac{x}{\sqrt{2}}\right)+
     c_5\exp\left(\frac{x}{\sqrt{2}}\right)\right)\,,\\
\psi_1:=&\frac{x}{c_6^{1/4}}+c_7\,,\quad a_1:=c_6\,,\quad b_1:=0\,,
\end{split}
\]
where $c_k$ $(k\in \{1,\ldots,7\})$ are arbitrary constants and the subscript 1 represents the obtained solution
to the differential system $S_1$. Ineq.~\eqref{IneqS1} imply $c_1\neq 0$ and $c_6\neq 0$.

The second differential system $S_2$ is generic, and its set of equations is given by
\begin{eqnarray}\label{EqS2}
&& 32\,y\,a\phi_y-64\,a^3(\phi+\phi_{xxxx})-96\,a^2a_x\phi_{xxx} \nonumber \\
&&\quad-36\,a\,a_x^2\phi_{xx}-3\,a_x^3\phi_x=0\,,\nonumber \\
&&128\,x^2a^3\phi_{xxxy}+15\,x^2a_x^3\psi_y-144\,x\,a\,a_x^2\phi_{y} \nonumber\\
&&\quad +288\,a^2a_x\phi_y=0\,, \nonumber\\
&&16\,x^2a^2\phi_{xxy}-3\,x^2a_x^2\phi_y+24\,x\,a\,a_x\phi_y \nonumber\\
&&\quad -32\,a^2\phi_y=0\,,\\
&&x\,a\,\phi_{xy}+3\,x\,a_x\phi_y-16\,a\phi_y=0\,,\nonumber \\
&&y\phi_{yy}-\phi_y=0,\ a\,\psi_x^4-1=0\,,\ \psi_y=0\,,\nonumber \\
&&8\,a\,a_{xx}-7\,a_x^2=0,\ a_y=0\,,\nonumber \\
&&b=0,\ h-8\,x^2=0\,. \nonumber
\end{eqnarray}

\begin{table*}[!ht]
\centering
\caption{CPU times (sec.)}
\bg {|c|c|c|c|c|c|c|c|c|c|c|c|c|c|} \hline\hline
Test & \multicolumn{13}{c|}{Order $n$ of ODE~\eqref{y2}\ \ \ \ \ \ \ \ \ \ \ \ }
\\ \cline{2-14}
     & 3 & 4 & 5 & 6 & 7 & 8 & 9 & 10 & 11 & 12 & 13 & 14 & 15 \\
\hline
  I & 0.20 & 0.61 & 1.27  & 2.54 & 4.18    & 6.49  & 10.20   & 23.21 & 39.79   & 63.38  & 91.54 & 119.42 & 150.13 \\
\hline
 \ \,\,I$_{\text A}$ & 0.28  & 0.83  & 1.51  & 3.01  & 5.28  & 9.52  & 16.83  & 45.40 & 80.72 & 150.19  & 291.13  & 484.35  & 751.20  \\
\hline
 II & 0.65  & 2.33  & 13.28  & 80.76  & 376.40  &1525.1   &7512.9   & OOM  & OOM   & OOM  & OOM  & OOM  & OOM  \\
\hline \hline
\nd
\end{table*}

\noindent
The set of inequations in $S_2$ consists three elements:
\begin{equation}\label{IneqS2}
S_2^{\neq}=\{a\neq 0,\phi_y\neq 0,a_x\neq 0\}.
\end{equation}
The differential system~\eqref{EqS2} is also easily solvable, and its general solution reads
\begin{equation}\label{SolS2}
\begin{array}{l}
\phi_2:=\frac{\phi_1}{(c_6x+c_7)^3}\,,\quad \psi_2:=c_8-\frac{1}{c_6(c_6x+c_7)}\,,\\[0.2cm]
a_2:=(c_6x+c_7)^8\,,\quad b_2:=0\,,\quad h:=8\,x^2\,.
\end{array}
\end{equation}
Here $\phi_1$ is the above presented solution to~\eqref{EqS1} for $\phi$, $c_i$ $(1\leq i\leq 8)$ are
arbitrary constants. The constraints that follow from~\eqref{IneqS2} are those in $S_1$, $c_1\neq 0, c_6\neq 0$,
and the additional inequation $c_6x+c_7\neq 0$ what rules out singularity in~\eqref{SolS2}. The obtained explicit
solutions to $S_1$ and $S_2$ form disjoined sets, since $a_x=0$ for a solution to $S_1$ and $a_x\neq 0$ for that
to $S_2$. The disjointness of solution sets for the output simple systems is guaranteed by the Thomas decomposition
algorithm~(\cite{Thomas'37}--\cite{Robertz'14}). In the given case a solution to $S_1$ provides a mapping
of~\eqref{IM} into the linear ODE
\[
   u^{(4)}(t)+c_6u(t)=0
\]
with constant coefficients, whereas a solution to $S_2$ maps~\eqref{IM} into an equation with variable coefficients
\[
   u^{(4)}(t)+(c_6x+c_7)^8u(t)=0\,.
\]
In~\cite{IbragimovMeleshkoSuksern'08}, the simplest form of the linearizing
transformation~\eqref{transformation} was found:
\[
    t=x\,,\quad u=x^2y^2\,,
\]
which maps~\eqref{IM} and \eqref{ParameterValue} into $u^{(4)}+u=0$ and corresponds to the solution of $S_1$
with
\[
c_1=1,\quad c_2=c_3=c_4=c_5=c_7=0\,.
\]
\label{4thOrder}
\end{example}

\begin{example}
As a serial example, we consider
\begin{equation}\label{y2}
(y^2)^{(n)} + y^2 =0\,,\ n\in \N_{\geq 3}\,.
\end{equation}
Obviously, Eq.~\eqref{y2} becomes $u^{(n)} + u =0$ via transformation~\eqref{transformation} of the form $t=x$ and $u=y^2$. We use this example as a benchmark for a comparative experimental analysis of the time behavior of our
algorithms when the order of the ODE grows. Additionally, we measure the CPU time for the
algorithm {\textsl{\bfseries{LinearizationTest I}}} whose subalgorithm {\textsl{\bfseries{DerivedAlgebra}}}
(in line 9) is replaced with the Maple implementation~\cite{Ceballos'12} for the detection of
an $n$-dimensional abelian subalgebra of the Lie symmetry algebra (line 8). By Thm.~\ref{theorem1},
the existence of such a subalgebra yields the criterion of linearization. Table 1 presents the CPU times, where ``OOM'' is an acronym for runs ``Out Of Memory''. The timings in
the table correspond to {\textsl{\bfseries{LinearizationTest I}}} (upper row), to its above described modification denoted by I$_{\text{A}}$ (middle row), and to {\textsl{\bfseries{LinearizationTest II}}}
(bottom row). As one can see, our first test (I) is the fastest and the second test (II) is the slowest. However,
the last one, unlike the other two, outputs much more information on the linearization. This fact was illustrated by Example~\ref{4thOrder}.

\label{ExSerial}
\end{example}

\section{Conclusions}

For the first time, the problem of the linearization test for a wide class of ordinary
differential equation of arbitrary order was algorithmically solved. In doing so, we have restricted
ourselves to the quasi-linear equations with a rational dependence on the other variables and to
point transformations, and designed two algorithmic tests in order to check linearizability. The main benefits of these restrictions are (i) the algorithmic construction of the Lie symmetry algebra for the input equation
and (ii) the reduction of the number of coefficients in the linearized equation due to the Lagerre-Forsyth canonical form~\eqref{lf-form}.

The benefit (i) allowed us to design an efficient algorithm {\textsl{\bfseries{LinearizationTest I}}}
which checks the linearizability of the equations. The second benefit (ii)
provides the feasibility of the algorithm {\textsl{\bfseries{LinearizationTest II}}} because of the overdeterminacy~(cf.~\cite{Seiler'10}, Sect.~7.5) of a linearizing differential system.
This overdeterminacy simplifies the consistency analysis of the linearizing system
answering the same question as the first test.

Moreover, due to finite-dimensionality of the solution space (cf.~\cite{Olver'95}, Prop.~6.57) of a
linearizing system, the Thomas decomposition algorithm outputs overdetermined
subsystems, as those in Example~\ref{4thOrder}. In practice, the overdeterminacy of
the outputted simple systems of the Thomas decomposition of a linearizing differential system makes them
easily solvable, much like the determining systems in the Lie symmetry analysis. Thereby, with
the algorithm {\textsl{\bfseries{LinearizationTest~II}}} one can not only detect linearizability, but
also find the linearizing transformation~\eqref{transformation} and the coefficients in
the linear form of the equation.

The Thomas decomposition for linearizing differential systems, even in the case of its inconsistency,
may be time and space consuming, especially for higher-order ODEs. That is why, in practice, it is advisable
to check the linearizability of the equation under consideration by the first algorithm before applying
the second one. In the case when Eq.~\eqref{ode} contains
parameters and/or arbitrary functions, there is no choice and one has to use the second algorithm.

The second algorithm may also improve the built-in Maple solver {\em dsolve} of differential equations. For example, {\em dsolve} applied to equation
\begin{equation}\label{Ibr6.6.57}
y^{\prime\prime\prime}+\frac{3y^\prime}{y}(y^{\prime\prime}-y^\prime)-3y^{\prime\prime}+2y^\prime-y=0
\end{equation}
outputs its solution implicitly in the complicated form of double integrals including the Maple
symbolic presentation {\em RootOf} for the roots of expressions. On the other hand Eq.~\eqref{Ibr6.6.57}
admits the linearization (cf.~\cite{Ibragimov'09}, Eqs.~6.6.57--6.6.59)
\[
   u^{\prime\prime\prime}-\frac{2}{t^3}u=0\,,\quad t=\exp(x)\,,\quad u=y^2\,
\]
which is easily obtained by our algorithm~{\textsl{\bfseries{LinearizationTest II}}} and provides the explicit form
of the solution to~\eqref{Ibr6.6.57}.

\section{Acknowledgments}
The authors are grateful to Daniel Robertz and Boris Dubrov for helpful discussions and to the anonymous reviewers for several insightful comments that led to a substantial improvement of the paper.

This work has been partially supported by the King Abdullah University of Science and Technology (KAUST baseline funding; D.~A.~Lyakhov and D.~L.~Michels), by the Russian Foundation for Basic Research (grant No.16-01-00080; V.~P.~Gerdt), and by the Ministry of Education and Science of the Russian Federation (agreement 02.a03.21.0008; V.~P.~Gerdt).

\end{document}